\documentclass[12pt]{amsart}
\usepackage{enumerate}
\usepackage{amsmath, amssymb, amsthm, amsfonts}
\usepackage{tikz}
\usepackage{geometry}
\usepackage[pdftex, bookmarks=true]{hyperref}
\usepackage{bbm} 
\usepackage{url}
\usepackage{cmap} 
\usepackage{verbatim}  
\usepackage{bm}  

\usepackage{marginnote}

\newtheorem{theorem}{Theorem}[section]

\newtheorem{corollary}[theorem]{Corollary}
\newtheorem{lemma}[theorem]{Lemma}

\newtheorem*{claim*}{Claim}

\theoremstyle{definition}

\newtheorem*{remark*}{Remark}

\newtheorem{definition}[theorem]{Definition}

\newcommand\maxcut{\texttt{MaxCut} }

\newcommand\Eb{\mathbb{E}}  
\newcommand\Pb{\mathbb{P}}  

\newcommand\Zb{\mathbb{Z}}
\newcommand\Nb{\mathbb{N}}
\newcommand\Gb{\mathbb{G}}
\newcommand\Gconf{\mathbb{G}^{\mathrm{conf}}}

\newcommand\Fc{\mathcal{F}}

\newcommand\Hc{\mathcal{H}}

\newcommand\Pc{\mathcal{P}}  

\newcommand\qb{\bm{q}}

\newcommand\ab{\bm{a}}

\newcommand\al{\alpha}
\newcommand\be{\beta}
\newcommand\ga{\gamma}

\newcommand\ka{\kappa}

\newcommand\si{\sigma}

\newcommand\Fst{F^{\mathrm{star}}}
\newcommand\Fed{F^{\mathrm{edge}}}
\newcommand\Fbst{\mathfrak{F}_\beta^{\mathrm{star}}}
\newcommand\Fbed{\mathfrak{F}_\beta^{\mathrm{edge}}}
\newcommand\phist{\varphi^{\mathrm{star}}}
\newcommand\phied{\varphi^{\mathrm{edge}}}
\newcommand\Phist{\Phi^{\ast}}

\DeclareMathOperator{\sign}{sign}

\newcommand{\defeq}{\mathrel{\vcenter{\baselineskip0.5ex \lineskiplimit0pt
                     \hbox{\scriptsize.}\hbox{\scriptsize.}}}%
                     =}

\DeclareFontFamily{U}{matha}{\hyphenchar\font45}
\DeclareFontShape{U}{matha}{m}{n}{
  <-6> matha5 <6-7> matha6 <7-8> matha7
  <8-9> matha8 <9-10> matha9
  <10-12> matha10 <12-> matha12
  }{}

\DeclareSymbolFont{matha}{U}{matha}{m}{n}
\DeclareMathSymbol{\Lt}{3}{matha}{"CE}

\title{RSB bounds on the maximum cut}

\author{Viktor Harangi}
\address{HUN-REN Alfr\'ed R\'enyi Institute of Mathematics, Budapest, Hungary}
\email{harangi@renyi.hu}
\thanks{The author was supported by the MTA-R\'enyi Counting in Sparse Graphs ``Momentum'' Research Group, by NRDI 
(grant KKP 138270), and by the Hungarian Academy of Sciences (J\'anos Bolyai Scholarship).}

\begin{document}

\begin{abstract}
In the context of random regular graphs, the size of the maximum cut is probably the second most studied graph parameter after the independence ratio. Zdeborová and Boettcher used the cavity method, a non-rigorous statistical physics technique, to predict one-step replica symmetry breaking (1-RSB) formulas. Coja-Ohglan et al. confirmed these predictions as rigorous upper bounds using the interpolation method. While these upper bounds were not expected to be exact, they may be very close to the true values.

In this paper, we establish 2-RSB upper bounds and fine-tune their parameters to beat the aforementioned 1-RSB bounds.
\end{abstract}


\maketitle

\section{Introduction} \label{sec:intro}

This paper is concerned with the size of the \emph{maximum cut} of \emph{random regular graphs}. 

Recall that a \emph{cut} of a graph $G$ is a partition of the vertex set $V(G)$ into two sets: $V_1 \sqcup V_2=V(G)$. Let $e[V_1,V_2]$ denote the \emph{size of the cut}, that is, the number of edges crossing the cut (i.e., going between $V_1$ and $V_2$). We will write $\ga(G)$ for the size of the maximum cut normalized by the total number of edges in $G$:
\begin{equation} \label{eq:gamma_G}
\ga(G) \defeq \max_{V_1 \sqcup V_2 = V(G)} \frac{e[V_1,V_2]}{e(G)} .
\end{equation}

Furthermore, a graph is said to be \emph{regular} if each vertex has the same degree. For a fixed degree $d$, let $\Gb_{N,d}$ be a uniform random $d$-regular graph on $N$ vertices. This paper provides improved upper bounds on $\ga\big( \Gb_{N,d} \big)$ that hold asymptotically almost surely as $N \to \infty$.

\subsection{Ising model}
A well-established approach for estimating the size of the maximum cut is to study the free energy of the Ising model. On a finite graph $G$, the Ising model is defined as follows. The \emph{Hamiltonian} of a \emph{spin configuration} 
$\tau=\big( \tau_v \big)_{v \in V(G)} \in \{-1,+1\}^{V(G)}$ is 
\[ \Hc_G(\tau) \defeq \sum_{uv \in E(G)} 
\frac{1+\tau_u \tau_v}{2} ;\]
i.e., $\Hc_G(\tau)$ is simply the number of edges connecting vertices of the same spin. If $G$ is $d$-regular and has $N$ vertices, then we have   
\[ \Hc_G(\tau) = \frac{Nd}{2} - e[V_\tau^-,V_\tau^+] 
\text{, where }
V_\tau^- \defeq \{ v \, : \, \tau_v=-1 \} 
\text{ and } 
V_\tau^+ \defeq \{ v \, : \, \tau_v=+1 \} .\]
Here $e[V_\tau^-,V_\tau^+]$ is the size of the cut corresponding to $\tau$. Consequently, on regular graphs, minimizing the Hamiltonian $\Hc_G$ of the Ising model is equivalent to finding the maximum cut.

Given a real parameter $\be$, the partition function of the model is 
\[ Z_{G,\be} \defeq \sum_{\tau \in \{\pm 1\}^{V(G)}} 
\exp\big( -\be  \Hc_G(\tau) \big) .\]
In the antiferromagnetic regime $\be>0$, the corresponding Boltzmann distribution is biased towards opposite spins on neighboring vertices. In fact, we will be interested in the limit $\be \to \infty$, where the Boltzmann distribution concentrates on configurations corresponding to maximum cuts of $G$. 

A simple but key observation is that for $\be>0$ we have 
\[ Z_{G,\be} \geq 
\exp\bigg( - \be \min_\tau H_G(\tau) \bigg) 
= \exp\left( -\beta \big(1-\ga(G) \big) \frac{Nd}{2} \right) .\]
Equivalently,
\begin{equation} \label{eq:gamma_bound}
\ga(G) \leq 1+\frac{\log Z_{G,\be}}{\be Nd/2} .
\end{equation}

If we consider the random $d$-regular graph $\Gb=\Gb_{N,d}$ on $N$ vertices, then the following limit is known\cite{bayati2013combinatorial} to exist for every $d,\be$: 
%
\begin{equation} \label{eq:Phi_def}
\Phi_d(\be) \defeq 
\lim_{N \to \infty} \frac{\Eb \log Z_{\Gb,\be}}{N} .
\end{equation} 
Furthermore, let 
\[ \Phist_d \defeq \liminf_{\be \to \infty} \frac{\Phi_d(\be)}{\be} .\]
The following asymptotic bound follows easily from \eqref{eq:gamma_bound} and the well-known fact that $\ga(\Gb_{N,d})$ is tightly concentrated around its mean; see Section~\ref{sec:concentration} for details. 
\begin{lemma} \label{lem:gamma_general_bound} 
It holds with probability $1-o_N(1)$ that 
\[ \ga\big( \Gb_{N,d} \big) \leq 1+ \frac{2}{d} \Phist_d + o_N(1) .\]
\end{lemma}
Consequently, any upper bound on $\Phist_d$ immediately yields an upper bound for the \maxcut problem over random $d$-regular graphs.

\subsection{Interpolation method}
The interpolation method is a rigorous technique for bounding $\Eb \log Z$ in various models. It is based on an ``interpolating'' family, which continuously transforms the original model into a ``decoupled'' one in such a way that $\Eb \log Z$ increases along the way. It has been used with great success to confirm a number of statistical physics predictions (non-rigorously) derived from the cavity method. In the context of maximum cuts, Coja-Oghlan et al. \cite{cojaoghlan2022ising} used the interpolation method to rigorously prove upper bounds predicted by Zdeborová and Boettcher \cite{zdeborova2010maxcut}. 

Originally, Franz and Leone established the interpolation method in sparse random graphs for a family of spin systems (including the Ising model) in \cite{franz2003replica}. This was later adapted to random regular graphs as well. A detailed exposition can be found in \cite[Appendix E]{sly2022sat}. The parameters of the general RSB bound include a \emph{functional order parameter} which can be thought of as a measure. In \cite{cojaoghlan2022ising} the authors used a simple and natural choice for this measure before taking the limit $\be \to \infty$. The resulting formula is far more manageable as it only has two free variables, see \eqref{eq:1rsb}. 

In this paper we consider 2-RSB formulas. At the ``deepest layer'' of the RSB hierarchy we employ a similar parameter choice as Coja-Oghlan et al. before taking the limit $\be \to \infty$. The resulting bound still has a measure parameter. Interestingly, even purely atomic measures with a small number of atoms yield good results. The key is that both the size and the location of each atom are treated as variables. These variables can be tuned to a great precision due to the relatively small number of variables. 

See Theorem~\ref{thm:2rsb} below for the discrete 2-RSB formula we obtained. Table \ref{table:bounds} shows the best bounds we found via numerical optimization for specific values of $d$. While the improvements may seem small, we actually expect the true values to be fairly close to these upper bounds. 

\begin{table}[h!] 
\caption{Upper bounds on the \maxcut problem for $3 \leq d \leq 7$: previous 1-RSB bounds and our new bounds on $\ga(\Gb_{N,d})$}
\centering
\begin{tabular}{l||l|l|l|l|l} 
$\deg$ ($d$) & 3 & 4 & 5 & 6 & 7 \\ 
\hline
1-RSB \cite{cojaoghlan2022ising} & 
0.92410 & 0.86823 & 0.83504 & 0.80486 & 0.78509 \\ 
$1^+$-RSB &
0.92403 & 0.86820 & 0.83494 & 0.80480 & 0.78499 \\
2-RSB &
0.92386 & 0.86790 & 0.83463 & 0.80435 & 0.78462 
\end{tabular}
\label{table:bounds}
\end{table}

Using a similar approach, in \cite{harangi2023rsb} we obtained RSB bounds on the independence ratio of $\Gb_{N,d}$ via the interpolation method applied to the hard-core model, improving on the previous best 1-RSB bounds of Lelarge and Oulamara \cite{lelarge2018replica}.

\subsection{Discrete 2-RSB formula}
Our 2-RSB upper bound includes a number of parameters: real numbers $z,m \in (0,1)$ and probabilities $p_i$ along with three-dimensional vectors $\qb_i$. Given a vector $\ab \in [0,1]^3$, we will use the following notation for its coordinates:
\[ \ab=\big( a^-,a^0,a^+ \big) .\]
We will always assume that the three coordinates add up to one: $a^- + a^0 + a^+ = 1$.

The formula involves two functions ($\Fed_z$ and $\Fst_z$) whose arguments are three-dimensional vectors. Specifically,
\begin{equation} \label{eq:Fed}
\Fed_z(\ab_1,\ab_2) \defeq 1 + 
(z-1)\big( a_1^- a_2^- + a_1^+ a_2^+ \big) .
\end{equation}
To define the other function, we need the following notation. Given a sequence $\si=(\si_1,\ldots,\si_d)$ of signs $\{ -,0,+ \}$, let $k^-(\si)$, $k^0(\si)$, and $k^+(\si)$ denote the number of $-$'s, $0$'s, and $+$'s in $\si$, respectively. Furthermore, let $k(\si) \defeq \min \big( k^-(\si),k^+(\si) \big)$. Then we define $\Fst_z$ as follows: 
\begin{equation} \label{eq:Fst}
\Fst_z(\ab_1,\ldots, \ab_d) \defeq 
\sum_{\si \in \{-,0,+\}^d } 
a_1^{\si_1} \cdots a_d^{\si_d} z^{k(\si)} .
\end{equation}
Now we can state our upper bound.
\begin{theorem} \label{thm:2rsb}
Let $z,m \in (0,1)$ and $n \in \Nb$. Suppose that\footnote{The condition $p_1+\cdots+p_n=1$ may be omitted because multiplying each $p_i$ with the same constant does not change the value of the upper bound formula.} 
\[ p_1, \ldots, p_n \in [0,1] \text{ with } 
p_1 + \cdots + p_n = 1 ,\]
and that for each $i=1,\ldots,n$ we have 
\[ \qb_i = \big( q_i^-,q_i^0,q_i^+ \big) \in [0,1]^3 
\text{ with } 
q_i^- + q_i^0 + q_i^+ = 1 .\] 
Then 
\begin{align*} 
\Phist_d \, m \log(1/z) &\leq 
\log \sum_{i_1=1}^n \cdots \sum_{i_d=1}^n \, 
\bigg( \prod_{\ell=1}^d p_{i_\ell} \bigg) 
\bigg( \Fst_z\big( \qb_{i_1},\ldots,\qb_{i_d}\big) \bigg)^m \\ 
&- \frac{d}{2} \log \sum_{i_1=1}^n \sum_{i_2=1}^n 
p_{i_1} p_{i_2} 
\bigg( \Fed_z\big( \qb_{i_1},\qb_{i_2}\big) \bigg)^m .
\end{align*}
\end{theorem}

Let us write up $\Fst$ more explicitly. We may partition the sequences $\si$ according to the value $k(\si)$: for $0 \leq k \leq d/2$ let 
\[ S_{d,k} \defeq \bigg\{ 
\si \in \{-,0,+\}^d \, : \, 
k(\si)=k \bigg\} ,\]
and define the corresponding polynomials:  
\[ P_{d,k}(\ab_1,\ldots, \ab_d) \defeq 
\sum_{\si \in S_{d,k}} a_1^{\si_1} \cdots a_d^{\si_d} .\]
Then $\Fst_z$ can be written as
\[ \Fst_z(\ab_1,\ldots, \ab_d) 
= \sum_{k} P_{d,k}(\ab_1,\ldots, \ab_d) z^k .\]
Note that each $P_{d,k}$ is a homogeneous multivariate polynomial (in $3d$ variables) of degree $d$. Adding them up for all $k$ gives $1$:
\[ \sum_{k} P_{d,k}(\ab_1,\ldots, \ab_d) 
= \sum_{\si \in \{-,0,+\}^d } 
a_1^{\si_1} \cdots a_d^{\si_d}
= \prod_{i=1}^d \big( a_i^- + a_i^0 + a_i^+ \big) 
= 1^d = 1. \]
For a specific example, consider the smallest case $d=3$, when $k(\si)$ has two possible values ($0$ and $1$), and it is easy to see that $S_{3,1}$ consists of the following $12$ triples:
\begin{itemize}
\item $(-,-,+)$ and its permutations ($3$ triples);
\item $(-,+,+)$ and its permutations ($3$ triples);
\item $(-,\,0,\,+)$ and its permutations ($6$ triples).
\end{itemize}
Since $P_{3,0}(\ab_1,\ab_2,\ab_3) + P_{3,1}(\ab_1,\ab_2,\ab_3)=1$, it follows that 
\[ \Fst_z(\ab_1,\ab_2,\ab_3) 
= 1+(z-1) P_{3,1}(\ab_1,\ab_2,\ab_3)
= 1+(z-1) \sum_{\si \in S_{3,1}}
a_1^{\si_1} a_2^{\si_2} a_3^{\si_3} .\]

\subsection{Tuning the parameters}
Note that in Theorem~\ref{thm:2rsb} we get an upper bound for \underline{any} choice of the parameters $z,m,p_i,\qb_i$. One can use a computer to find good substitutions for these parameters. It turns out that the function to optimize is very ``rugged'': it has a large number of local minima with values close to the global minimum. Therefore, basic (local) optimization techniques are not sufficient. Our approach was to first use a Particle Swarm optimization algorithm (attempting to explore this rugged landscape), then, starting from the best point found by the Particle Swarm, to perform a local optimization in order to fine-tune the parameters.

We used the formula for $n=4$.\footnote{Increasing $n$ considerably slows down the process, offering only minor improvements in return.} In order to reduce the number of free variables to optimize, we assumed that the parameters have the following symmetric structure:
\begin{align*}
p_i   &: \quad (p_1,p_1,1/2-p_1,1/2-p_1) ;\\
q_i^- &: \quad (q_1^-,q_2^-,q_3^-,q_4^-) ;\\
q_i^+ &: \quad (q_2^-,q_1^-,q_4^-,q_3^-) ;\\
q_i^0 &: \quad (q_1^0,q_1^0,q_3^0,q_3^0)  
\text{, where } 
q_1^0=1-q_1^- - q_2^-; \; q_3^0=1-q_3^- - q_4^-  .
\end{align*}
Therefore, we had the following seven variables to tune:
\[ z,m,p_1,q_1^-,q_2^-,q_3^-,q_4^- .\]
We made our program codes (SageMath) available on GitHub:
\begin{center}
\url{https://github.com/harangi/rsb/blob/main/maxcut.sage}
\end{center}
Table \ref{table:parameter_choices} shows the best substitutions we found for each $d=3,4,5,6,7$.

\begin{table}[h!] 
\caption{Parameter choices for our 2-RSB upper bounds for $3 \leq d \leq 7$. }.
\centering
\begin{tabular}{l|l||l|l|l|l|l|l|l}
$d$ & bound & $z$ & $m$ & $p_1$ & $q_1^-$ & $q_2^-$ & $q_3^-$ & $q_4^-$ \\
\hline
\textbf{3} & 0.923867 & 
0.20304&0.27824&0.20662&0.48970&0.14372&0.93831&0.02348 \\
\textbf{4} & 0.867901 & 
0.29823&0.34695&0.24571&0.68439&0.29092&0.95728&0.03657 \\
\textbf{5} & 0.834637 & 
0.28496&0.29369&0.17626&0.54693&0.21655&0.94913&0.02047 \\
\textbf{6} & 0.804356 & 
0.33817&0.31152&0.19822&0.65668&0.31416&0.95681&0.03662 \\
\textbf{7} & 0.784629 &
0.34133&0.29489&0.17105&0.58283&0.24269&0.95452&0.01989
\end{tabular}
\label{table:parameter_choices}
\end{table}

\subsection{Comparison with 1-RSB}
We mention that we can use the functions $\Fed_z$ and $\Fst_z$ to express the original 1-RSB bound from \cite{cojaoghlan2022ising}. Consider the following vector  
\[ \ab \defeq \big( \al, 1-2 \al, \al \big) \]
parametrized by the variable $\al \in (0,1/2)$. Then the 1-RSB bound can be written as 
\begin{equation} \label{eq:1rsb}
\Phist_d \log(1/z) \leq 
\log \Fst_z(\underbrace{\ab,\ldots,\ab}_{d}) 
- \frac{d}{2} \log \Fed_z(\ab,\ab) .
\end{equation}
According to \eqref{eq:Fed} we have  
$\Fed_z(\ab,\ab) = 1+(z-1)(a^- a^- + a^+ a^+) = 1+(z-1)2\al^2$, while $\Fst_z(\ab,\ldots, \ab)$ is a multivariate polynomial in $z$ and $\al$ that can be worked out for any specific value of $d$ using \eqref{eq:Fst}.\footnote{There is an explicit formula in \cite{cojaoghlan2022ising} (see eq.~1.8) that expresses $\Fst_z(\ab,\ldots, \ab)$ using the $d$-th power of a cleverly chosen sparse matrix.} The 1-RSB formula follows after minimizing the right-hand side of \eqref{eq:1rsb} in $z$ and $\al$.

%
%

\subsection{1-RSB revisited}
In fact, it is possible to improve upon the best previous bounds from \cite{cojaoghlan2022ising} even within the framework of one-step replica symmetry breaking by making a more elaborate choice in the parameters before taking the limit $\be \to \infty$. We will refer to these bounds as $1^+$-RSB bounds; see Section~\ref{sec:tweaked1rsb} for details.

Although the bounds we get this way are inferior to the ones obtained by 2-RSB, there are a smaller number of free variables to optimize. Consequently, one may be able to perform the necessary numerical optimization even for those values of $d$ that are too large for the 2-RSB approach.

%

%

%

\section{Configuration model and concentration} \label{sec:concentration}

This short section contains some standard and well-known facts about random regular graphs and their maximum cut.

Recall that we defined $\Gb_{N,d}$ as a uniform random graph among all $d$-regular simple graphs on the vertex set $\{1,\ldots,N\}$. As usual, we will work with the closely related random graphs produced by the \emph{configuration model} instead. Given $N$ vertices, each with $d$ ``half-edges'', the configuration model picks a random pairing of these $Nd$ half-edges, producing $Nd/2$ edges. The resulting graph is $d$-regular but it may have loops and/or multiple edges. We denote this random multigraph by $\Gconf_{N,d}$. A well-known fact is that if $\Gconf_{N,d}$ is conditioned to be simple, then we get back $\Gb_{N,d}$. Moreover, for any $d$, the probability that $\Gconf_{N,d}$ is simple converges to a positive $p_d$ as $N \to \infty$. 

We say that a property holds \emph{asymptotically almost surely (a.a.s.)} as $N \to \infty$ if it holds with probability $1-o_N(1)$. If $\Gconf_{N,d}$ a.a.s.~has a certain property, then so does $\Gb_{N,d}$. Therefore, it suffices to prove our a.a.s.~results for $\Gconf_{N,d}$.

In order to prove Lemma~\ref{lem:gamma_general_bound} we need an Azuma-type concentration result for $\ga(\Gb_{N,d})$. One may derive one from the following general result from Wormald's survey \cite{wormald1999survey}.
\begin{lemma} \label{lem:concentration}
\cite[Theorem 2.19]{wormald1999survey}
Let $\varphi$ be a real-valued graph parameter defined for $N$-vertex, $d$-regular multigraphs. A \emph{simple swithching} is replacing a pair of edges $v_1v_2$, $v_3v_4$ of a multigraph with the pair $v_1v_3$,$v_2v_4$. Suppose that $|\varphi(G')-\varphi(G)|\leq c$ holds whenever $G'$ can be obtained from $G$ using a simple switching. Then for any $t>0$ it holds for $\Gb=\Gconf_{N,d}$ that 
\[ \Pb\bigg( |\varphi(\Gb) - \Eb \varphi(\Gb)| \geq t \bigg) 
\leq 2 \exp\left( \frac{-t^2}{Ndc^2} \right) .\]
\end{lemma}
%
%
\begin{proof}[Proof of Lemma~\ref{lem:gamma_general_bound}]
Since a simple switching changes the size of the maximum cut by at most $2$, we can apply Lemma~\ref{lem:concentration} for $\varphi=\ga$ and $c=4/Nd$. Setting $t=(\log N)/\sqrt{N}$ we get that it holds with probability $1-o_N(1)$ that 
\[ \ga(\Gb) \leq \Eb \ga(\Gb) + \frac{\log N}{\sqrt{N}} 
\stackrel{\eqref{eq:gamma_bound}}{\leq} 
1+ \frac{\Eb \log Z_{\Gb,\be}}{\be Nd/2} + \frac{\log N}{\sqrt{N}}
\stackrel{\eqref{eq:Phi_def}}{=} 
1+ \frac{2}{d} \frac{\Phi_d(\be)}{\be} + o_N(1) \]
for any given $\be>0$. Then the same is true if we replace $\Phi_d(\be)/\be$ with $\Phi^\ast_d=\liminf \Phi_d(\be)/\be$.
\end{proof}
%

\section{Replica formulas} \label{sec:formulas}

Originally the cavity method and belief propagation were non-rigorous techniques in statistical physics to predict the free energy of various models. They inspired a large body of rigorous work, and over the years several predictions were confirmed. In particular, the so-called \emph{interpolation method} has been used with great success to establish rigorous upper bounds on the free energy. 

Building on the pioneering works \cite{franz2003replica,franz2003replica_non-poissonian,panchenko2004bounds}, there is a rigorous exposition of the interpolation method in the context of random regular graphs in \cite[Appendix E]{sly2022sat}. It is carried out for a general class of models. They mention two chief examples, one of them being the Potts model (a generalization of the Ising model with $q \geq 2$ spin states). A similar exposition had been given in \cite{lelarge2018replica}. Other sources explaining the interpolation method for specific models include \cite{ayre2022lower} and \cite[Section 5]{harangi2023rsb}.

\subsection{The general replica bound} \label{sec:r-rsb}
Here we present the general $r$-step RSB bound for the Ising model/\maxcut problem. For a topological space $\Omega$ let $\Pc(\Omega)$ denote the space of Borel probability measures on $\Omega$ equipped with the weak topology. We define $\Pc^k(\Omega)$ recursively as 
\[ \Pc^k(\Omega) \defeq 
\begin{cases}
\Pc( \Omega) & \text{if } k=1;\\
\Pc\big( \Pc^{k-1}( \Omega) \big) & \text{if } k \geq 2.
\end{cases}\]
The general bound will have the following parameters:
\begin{itemize}
\item $\be > 0$;
\item $0< m_1, \ldots, m_r <1$ corresponding to the so-called Parisi parameters;
\item a measure $\eta^{(r)} \in \Pc^r\big( [-1,1] \big)$.
\end{itemize}
\begin{definition} \label{def:recursive_sampling}
Given a fixed $\eta^{(r)} \in \Pc^r$, we choose (recursively for $k=r-1,r-2,\ldots,1$) a random $\eta^{(k)} \in \Pc^k$ with distribution $\eta^{(k+1)}$. Finally, given $\eta^{(1)}$ we choose a random $x \in [0,1]$ with distribution $\eta^{(1)}$. In fact, we will need $d$ independent copies of this random sequence, indexed by $\ell \in \{1,\ldots, d\}$. Schematically:
\[ \eta^{(r)} \, \to \, \eta_\ell^{(r-1)} \, \to \, \cdots \, \to \, 
\eta_\ell^{(1)} \, \to \, x_\ell \quad (\ell= 1,\ldots, d) .\] 

For $1 \leq k \leq r$ we define $\Fc_k$ as the $\sigma$-algebra generated by 
$\eta_\ell^{(r-1)}, \ldots, \eta_\ell^{(k)}$, $\ell=1,\ldots,d$, and by $\Eb_k$ we denote the conditional expectation w.r.t.\ $\Fc_k$. Note that $\Fc_r$ is the trivial $\sigma$-algebra and hence $\Eb_r$ is simply $\Eb$.

Given a random variable $V$ (depending on the variables $\eta_\ell^{(k)}, x_\ell$), let us perform the following procedure: raise it to power $m_1$, then apply $\Eb_1$, raise the result to power $m_2$, then apply $\Eb_2$, and so on. In formula, let $T_0 V \defeq V$ and recursively for $k=1, \ldots, r$ set 
\[ T_k V \defeq \Eb_k (T_{k-1} V)^{m_k} .\]
In this scenario, applying $\Eb_k$ means that, given $\eta_\ell^{(k)}$, $\ell=1,\ldots,d$, we take expectation in $\eta_\ell^{(k-1)}$, $\ell=1,\ldots,d$ (or in $x_\ell$ if $k=1$). 
\end{definition}

Now we are ready to state the $r$-RSB bound given by the interpolation method.
\begin{theorem} \label{thm:r-rsb}
Let $r \geq 1$ be a positive integer and $\be, m_1, \ldots, m_r, \eta^{(r)}$ parameters as described above. Let $x_\ell$ ($\ell=1,\ldots,d$) denote the random variables obtained from $\eta^{(r)}$ via the procedure in Definition \ref{def:recursive_sampling}. 

Furthermore, let 
\begin{align*}
\Fbed(x_1,x_2) &\defeq \frac{1-x_1x_2}{2} + \frac{1+x_1x_2}{2}\exp(-\be) ;\\
\Fbst(x_1,\ldots,x_d) &\defeq 
f_\be^-(x_1)\cdots f_\be^-(x_d) + 
f_\be^+(x_1)\cdots f_\be^+(x_d) 
\text{, where}\\
f_\be^-(x) &\defeq \frac{1+x}{2} + \frac{1-x}{2} \exp(-\be);\\
f_\be^+(x) &\defeq \frac{1-x}{2} + \frac{1+x}{2} \exp(-\be).
\end{align*}
Then we have the following upper bound for the free energy $\Phi_d(\be)$ of the Ising model on random $d$-regular graphs:
\begin{equation*}
\Phi_d(\be) \, m_1 \cdots m_r 
\leq \log T_r \Fbst(x_1,\ldots,x_d)
- \frac{d}{2} \log T_r \Fbed(x_1,x_2).
\end{equation*}
\end{theorem}
This was rigorously proved in \cite{sly2022sat}. They actually considered a more general setting incorporating a class of (random) models over a general class of random hypergraphs (including the model considered in this paper).

We should make a number of remarks at this point.
\begin{itemize}
\item Above we slightly deviated from the standard notation as the usual form of the Parisi parameters would be
\[ 0 < \hat{m}_1 < \cdots < \hat{m}_r < 1 ,\]
where $\hat{m}_k$ can be expressed in terms of our parameters $m_k$ as follows:
\[ \hat{m}_r = m_1; \quad 
\hat{m}_{r-1} = m_1 m_2; \quad 
\ldots; \quad 
\hat{m}_{1} = m_1 m_2 \cdots m_r .\]
As a consequence, the indexing of $\Fc_k$, $\Eb_k$, $T_k$ is in reverse order, and the definition of $T_k$ simplifies a little because raising to power $1/\hat{m}_{r-k+2}$ and then immediately to $\hat{m}_{r-k+1}$ (as done, for example, in \cite{panchenko2004bounds}) amounts to a single exponent $\hat{m}_{r-k+1}/\hat{m}_{r-k+2}=m_k$ in our setting.
\item Also, generally there is an extra layer of randomness (starting from an $\eta^{(r+1)} \in \Pc^{r+1}$) resulting in another expectation outside the $\log$. This random choice is meant to capture the local structure of the graph in a given direction. However, random regular graphs have the same local structure around almost all vertices, and hence, in principle, we do not need this layer of randomness. Therefore, in the $d$-regular case one normally chooses a trivial $\eta^{(r+1)} = \delta_{\eta^{(r)}}$. That is why we omitted $\eta^{(r+1)}$ and started with a deterministic $\eta^{(r)}$.
\end{itemize}

\subsection{A specific choice} \label{sec:choice}
To make the formula in Theorem \ref{thm:r-rsb} more manageable, we make a specific choice at layer 1 (essentially the same as the one made in \cite{cojaoghlan2022ising} in the case $r=1$). We consider the limit $\be \to \infty$ and $m_1 \to 0$ in a way that $\be m_1$ stays constant and $x$ is concentrated on the three-element set $\{-1,0,+1\}$.\footnote{One could use atoms at $0$ and $\pm\big(1-\exp(-\be)\big)$, and the result would be the same. We mention this because this offers the possibility of a generalization: adding atoms at $\pm \big(1-\exp(-t\be) \big)$ for various values of $0<t<1$. This will be the key idea in Section~\ref{sec:tweaked1rsb}.} In other words, $\eta^{(1)}$ is a distribution 
$q^- \delta_{-1} + q^0 \delta_{0} + q^+ \delta_{+1}$ for some random $q^-,q^0, q^+ \in [0,1]$ with $q^- + q^0 + q^+=1$.

More specifically, fix some $0<z<1$ and consider the limit $\be \to \infty$, $m_1 \to 0$ with $\be m_1 = -\log z$. We will have to take the limit of various expressions of the form 
\[ \lim_{\be \to \infty} P\big( \exp(-\be) \big)^{m_1} 
\text{, where $P$ is a polynomial.}\]
It is easy to see that such a limit depends solely on the smallest exponent occurring in $P$. More precisely, suppose that 
\[ P\big( \exp(-\be) \big) = \sum_{k \in K} 
a_k \big( \exp(-\be) \big)^k ,\]
where $K$ is a finite set of exponents. (In fact, we do not need to assume that the exponents are integers.) Let $\ka \defeq \min K$ denote the smallest exponent and assume that the corresponding coefficient $a_\ka$ is positive. Then (keeping $\be m_1=-\log z$ fixed) we have 
\begin{equation} \label{eq:z_kappa}
\lim_{\be \to \infty} P\big( \exp(-\be) \big)^{m_1} 
= \lim_{\be \to \infty} \bigg( \underbrace{\exp(-\be m_1)}_{=z} \bigg)^\kappa 
\bigg( a_\ka \big(1+o_\be(1) \big) \bigg)^{m_1} 
= z^\kappa .
\end{equation}
It follows that, if each $x_\ell \in \{-1,0,+1\}$ was fixed, then 
\[ \Fbst(x_1,\ldots,x_d)^{m_1} \to z^{\min(k^-,k^+)} ,\]
where $k^-$ and $k^+$ denote the number of $-1$'s and $+1$'s among $x_1,\ldots,x_d$, respectively. 

Similarly,
\[ \Fbed(x_1,x_2)^{m_1} \to 
\begin{cases}
z & \text{if $x_1=x_2=-1$ or $x_1=x_2=+1$;} \\
1 & \text{otherwise.} 
\end{cases} .\]

Therefore, conditioned on 
\[ \eta_\ell^{(1)} 
= q_\ell^- \delta_{-1} 
+ q_\ell^0 \delta_{0} 
+ q_\ell^+ \delta_{+1} \]
for some deterministic 
$\qb_\ell=\big(q_\ell^-,q_\ell^0,q_\ell^+\big) \in [0,1]^3$ ($\ell=1,\ldots,d$), we get 
\begin{align*}
T_1 \Fbed(x_1,x_2) &\to \Fed_z(\qb_1,\qb_2) ;\\
T_1 \Fbst(x_1,\ldots,x_d) &\to \Fst_z(\qb_1,\ldots,\qb_d) ,
\end{align*}
where we defined $\Fed_z$ and $\Fst_z$ in \eqref{eq:Fed} and \eqref{eq:Fst}, respectively.

In the resulting formula the randomness in layer $1$ disappears along with the Parisi parameter $m_1$. After re-indexing ($k \to k-1$) we get the following corollary.
\begin{corollary} \label{cor:r-rsb}
Let $z \in (0,1)$ and $0< m_1, \ldots, m_{r-1} <1$. Furthermore, fix a deterministic $\pi^{(r-1)} \in \Pc^{r}\big( \{-1,0,+1\}\big)$ and take $d$ independent copies of recursive sampling:
\[ \pi^{(r-1)} \, \to \, \pi_\ell^{(r-2)} \, \to \, \cdots \, \to \, 
\pi_\ell^{(1)} \, \to \, \qb_\ell \quad (\ell= 1,\ldots, d) .\] 
We define the conditional expectations $\Eb_k$ and the corresponding $T_k$ as before, w.r.t.~this new system of random variables. Then 
\begin{equation*}
\Phist_d \, m_1 \cdots m_{r-1} \log(1/z) 
\leq \log T_{r-1} \Fst_z(\qb_1,\ldots,\qb_d) 
- \frac{d}{2} \log T_{r-1} \Fed_z(\qb_1,\qb_2) .
\end{equation*}
\end{corollary}
\begin{proof}
For a formal proof one needs to define an $\eta^{(r)}\in \Pc^r([-1,1])$ for the fixed $\pi^{(r-1)}$ such that the corresponding $\eta_\ell^{(1)}$ is distributed as $q_\ell^- \delta_{-1} + q_\ell^0 \delta_{0} + q_\ell^+ \delta_{+1}$. Then Theorem \ref{thm:r-rsb} can be applied and we get the new formula when taking the limit $\be \to \infty$. 
\end{proof}

In the special case when every distribution is purely atomic, we get discrete formulas providing upper bounds for $\Phist_d$, and hence, via Lemma~\ref{lem:gamma_general_bound}, for the \maxcut problem as well.

In the case $r=1$ we have a deterministic $\qb$. It is easy to see that we may assume that $\qb$ is symmetric in the sense that $q^-=q^+ \defeq \al$, and we get back \eqref{eq:1rsb}.

As for the case $r=2$, if $\pi^{(1)}$ is a purely atomic measure with $n$ atoms such that the $i$-th atom has measure $p_i$ and corresponds to the vector $\qb_i$, then we get the discrete 2-RSB bound stated in Theorem~\ref{thm:2rsb} of the introduction.

\section{Tweaked 1-RSB bounds} \label{sec:tweaked1rsb}

It is possible to improve the previous best bounds even within the framework of the $r=1$ case of the interpolation method. In order to achieve this, we need to choose the parameter measure in a more intricate manner.

Depending on $\be$, we define the following $[-1,1]\to [-1,1]$ mapping:
\[ \xi_\be(t) \defeq 
\sign(t) \bigg( 1 - \exp(-|t|\be) \bigg) =  
\begin{cases}
-1+\exp(-|t|\be) & \text{if } t<0 ;\\
0 & \text{if } t=0 ;\\
1-\exp(-|t|\be) & \text{if } t>0 .
\end{cases} \]
For fixed $t_1,\ldots,t_d \in [-1,1]$, we set $x_i \defeq \xi_\be(t_i)$ ($i=1,\ldots,d$) and let $\be$ converge to $\infty$. Furthermore, for some fixed $z \in (0,1)$, we set $m_1 \defeq (-\log z)/\be$ as before. Using \eqref{eq:z_kappa} it is easy to see that we obtain the following in the limit: 
\[ \Fbst(x_1,\ldots,x_d)^{m_1} \to z^{\phist(t_1,\ldots,t_d)} 
\text{ as } \be \to \infty ,\]
where the function $\phist \colon [-1,1]^d \to [0,\infty)$ in the exponent is defined as 
\[ \phist(t_1,\ldots,t_d) \defeq \min(k^-,k^+) \text{, where } k^- = \sum_{i:\, t_i<0} |t_i| \text{ and } 
k^+ = \sum_{i:\, t_i>0} |t_i| \]
are the sum of the absolute values of the negative and positive $t_i$'s, respectively. 

Similarly, for some fixed $(t_1,t_2) \in [-1,1]^2$, 
\[ \Fbed(x_1,x_2)^{m_1} \to z^{\phied(t_1,t_2)} 
\text{ as } \be \to \infty ,\]
where the exponent $\phied \colon [-1,1]^2 \to [0,1]$ is defined as 
\[ \phied(t_1,t_2) \defeq 
\begin{cases}
\min(|t_1|,|t_2|) & \text{if } t_1t_2>0 ;\\
0 & \text{otherwise} .
\end{cases}
\]
We can take any measure $\tau$ on $[-1,1]$ and define $\eta$ as the push-forward of $\tau$ w.r.t.\ the mapping $\xi_\be$. In the limit we get the following upper bound:
\[
\Phist_d \log(1/z) \leq 
\log\bigg( \int z^{\phist(t_1,\ldots,t_d)} \, \mathrm{d} \tau^d(t_1,\ldots,t_d) \bigg) \\
- \frac{d}{2} \log\bigg( \int z^{\phied(t_1,t_2)} \, \mathrm{d} \tau^2(t_1,t_2) \bigg) .
\]
Our strategy is that for some $M \in \Nb$ we take the finite set 
\[ T=\left\{ \frac{i}{M} \, : \, 
i \in \Zb \text{ and } -M \leq i \leq M \right\} \subset [-1,1] \]
along with a symmetric probability distribution on $T$ described by the probabilities $q_i=q_{-i}$ ($i=0,\ldots,M$). Then the upper bound is a relatively simple function\footnote{The functions inside both $\log$'s are of the form $\sum_{i \in \Zb} P_i(q_0,\ldots,q_M) z^{i/M}$ with each $P_i$ being a homogeneous multivariate polynomial (of degree $d$).} of $z,q_0,q_1,\ldots, q_M$, where we have the following constraints for the variables:
\[ z \in (0,1); \quad q_i \in [0,1]; \quad 
q_0 + 2\sum_{i=1}^M q_i = 1 .\]
%
For given $d$ and $M$, one may work out the coefficients of $z^{i/M}$ (as polynomials of $q_0,\ldots,q_M$), and use numerical optimization to find good values to substitute into the obtained upper bound formula. We carried this out in the simplest case $M=2$ and included the obtained bounds in Table~\ref{table:bounds}.


\bibliographystyle{plain}
\bibliography{refs}

\begin{thebibliography}{10}

\bibitem{ayre2022lower}
Peter Ayre, Amin Coja-Oghlan, and Catherine Greenhill.
\newblock Lower bounds on the chromatic number of random graphs.
\newblock {\em Combinatorica}, 42(5):617--658, 2022.

\bibitem{bayati2013combinatorial}
Mohsen Bayati, David Gamarnik, and Prasad Tetali.
\newblock {Combinatorial approach to the interpolation method and scaling limits in sparse random graphs}.
\newblock {\em The Annals of Probability}, 41(6):4080 -- 4115, 2013.

\bibitem{cojaoghlan2022ising}
Amin Coja-Oghlan, Philipp Loick, Bal\'{a}zs~F. Mezei, and Gregory~B. Sorkin.
\newblock The ising antiferromagnet and max cut on random regular graphs.
\newblock {\em SIAM Journal on Discrete Mathematics}, 36(2):1306--1342, 2022.

\bibitem{franz2003replica}
Silvio Franz and Michele Leone.
\newblock Replica bounds for optimization problems and diluted spin systems.
\newblock {\em Journal of Statistical Physics}, 111(3):535--564, May 2003.

\bibitem{franz2003replica_non-poissonian}
Silvio Franz, Michele Leone, and Fabio~Lucio Toninelli.
\newblock Replica bounds for diluted non-poissonian spin systems.
\newblock {\em Journal of Physics A}, 36:10967--10985, 2003.

\bibitem{harangi2023rsb}
Viktor Harangi.
\newblock Improved replica bounds for the independence ratio of random regular graphs.
\newblock {\em J. Stat. Phys.}, 190(3), March 2023.

\bibitem{lelarge2018replica}
Marc Lelarge and Mendes Oulamara.
\newblock Replica bounds by combinatorial interpolation for diluted spin systems.
\newblock {\em Journal of Statistical Physics}, 173(3):917--940, Nov 2018.

\bibitem{panchenko2004bounds}
Dmitry Panchenko and Michel Talagrand.
\newblock Bounds for diluted mean-fields spin glass models.
\newblock {\em Probability Theory and Related Fields}, 130(3):319--336, Nov 2004.

\bibitem{sly2022sat}
Allan Sly, Nike Sun, and Yumeng Zhang.
\newblock The number of solutions for random regular {NAE-SAT}.
\newblock {\em Probab. Theory Relat. Fields}, 182(1-2):1--109, February 2022.

\bibitem{wormald1999survey}
N.~C. Wormald.
\newblock Models of random regular graphs.
\newblock In {\em Surveys in combinatorics, 1999 ({C}anterbury)}, volume 267 of {\em London Math. Soc. Lecture Note Ser.}, pages 239--298. Cambridge Univ. Press, Cambridge, 1999.

\bibitem{zdeborova2010maxcut}
Lenka Zdeborová and Stefan Boettcher.
\newblock A conjecture on the maximum cut and bisection width in random regular graphs.
\newblock {\em Journal of Statistical Mechanics: Theory and Experiment}, 2010(02):P02020, feb 2010.

\end{thebibliography}

\end{document}